\documentclass{amsart}

\usepackage{amssymb,amsmath,stmaryrd,mathrsfs,amsthm}
\usepackage[all,2cell]{xy}
\usepackage{enumitem}
\usepackage{color}
\definecolor{darkgreen}{rgb}{0,0.45,0}
\usepackage[pagebackref,colorlinks,citecolor=darkgreen,linkcolor=darkgreen]{hyperref}
\usepackage{mathtools}          
\usepackage{url}                
\usepackage{xspace}             
\usepackage{cite}               
\usepackage{tikz}
\usetikzlibrary{arrows}
\newenvironment{tikzcenter}{\begin{center}\begin{tikzpicture}}{\end{tikzpicture}\end{center}}

\makeatletter
\let\ea\expandafter

\def\mdef#1#2{\ea\ea\ea\gdef\ea\ea\noexpand#1\ea{\ea\ensuremath\ea{#2}\xspace}}
\def\alwaysmath#1{\ea\ea\ea\global\ea\ea\ea\let\ea\ea\csname your@#1\endcsname\csname #1\endcsname
  \ea\def\csname #1\endcsname{\ensuremath{\csname your@#1\endcsname}\xspace}}

\mdef\sC{\mathscr{C}}
\mdef\sM{\mathscr{M}}
\mdef\cC{\mathcal{C}}
\mdef\cI{\mathcal{I}}
\mdef\fU{\mathfrak{U}}
\mdef\Hbar{\overline{H}}
\mdef\Util{\widetilde{U}}
\mdef\albar{\overline{\alpha}}

\alwaysmath{alpha}
\alwaysmath{beta}
\alwaysmath{gamma}
\alwaysmath{kappa}
\alwaysmath{lambda}
\let\al\alpha
\let\be\beta
\let\gm\gamma
\let\ka\kappa
\let\la\lambda

\mdef\bSet{\mathbf{Set}}
\mdef\bCat{\mathbf{Cat}}
\mdef\sSet{\mathbf{sSet}}
\mdef\nEq{\mathrm{Eq}}

\mdef\iscontr{\mathsf{isContr}}
\mdef\isprop{\mathsf{isProp}}
\mdef\isequiv{\mathsf{isEquiv}}
\mdef\equiv{\mathsf{Equiv}}
\mdef\fun{\mathsf{fun}}
\mdef\io{(\infty,1)}
\mdef\ig{\infty\mathsf{Gpd}}

\let\dn\downarrow
\newcommand{\op}{^{\mathrm{op}}}

\SelectTips{cm}{}
\newdir{ >}{{}*!/-10pt/@{>}}    
\newcommand{\pushoutcorner}[1][dr]{\save*!/#1+1.4pc/#1:(1,-1)@^{|-}\restore}
\newcommand{\pullbackcorner}[1][dr]{\save*!/#1-1.2pc/#1:(-1,1)@^{|-}\restore}

\DeclareMathOperator\colim{colim}

\newcommand{\too}[1][]{\ensuremath{\overset{#1}{\longrightarrow}}}

\let\toto\rightrightarrows
\let\into\hookrightarrow
\let\cof\rightarrowtail
\def\acof{\mathrel{\mathrlap{\hspace{3pt}\raisebox{4pt}{$\scriptscriptstyle\sim$}}\mathord{\rightarrowtail}}}
\let\fib\twoheadrightarrow
\def\afib{\mathrel{\mathrlap{\hspace{3pt}\raisebox{4pt}{$\scriptscriptstyle\sim$}}\mathord{\twoheadrightarrow}}}

\let\xto\xrightarrow
\def\twoheadrightarrowfill@{\arrowfill@{\relbar\joinrel\relbar}\relbar\twoheadrightarrow}
\newcommand\xtwoheadrightarrow[2][]{\ext@arrow 0055{\twoheadrightarrowfill@}{#1}{#2}}
\let\xfib\xtwoheadrightarrow

\def\toiso{\xto{\smash{\raisebox{-.5mm}{$\scriptstyle\sim$}}}}

\newif\ifhyperref
\@ifpackageloaded{hyperref}{\hyperreftrue}{\hyperreffalse}
  \ifhyperref
    \def\defthm#1#2{%
      \newtheorem{#1}{#2}[section]%
      \expandafter\def\csname #1autorefname\endcsname{#2}%
      \expandafter\let\csname c@#1\endcsname\c@thm}
  \else
    \def\defthm#1#2{\newtheorem{#1}[thm]{#2}}
    \ifx\SK@label\undefined\let\SK@label\label\fi
    \let\old@label\label
    \let\your@thm\@thm
    \def\@thm#1#2#3{\gdef\currthmtype{#3}\your@thm{#1}{#2}{#3}}
    \def\currthmtype{}
    \def\label#1{{\let\your@currentlabel\@currentlabel\def\@currentlabel%
        {\currthmtype~\your@currentlabel}%
        \SK@label{#1@}}\old@label{#1}}
    \def\autoref#1{\ref{#1@}}
  \fi

\theoremstyle{plain}
\newtheorem{thm}{Theorem}[section]

\defthm{lem}{Lemma}
\theoremstyle{remark}
\defthm{rmk}{Remark}

\numberwithin{equation}{section}

\@ifpackageloaded{mathtools}{\mathtoolsset{showonlyrefs,showmanualtags}}{}

\theoremstyle{plain}
\newtheorem*{rep@theorem}{\rep@title}
\newcommand{\newreptheorem}[2]{%
\newenvironment{rep#1}[1]{%
 \def\rep@title{Restatement of #2 \ref{##1}}%
 \begin{rep@theorem}}%
 {\end{rep@theorem}}}

\newreptheorem{lem}{Lemma}

\makeatother

\begin{document}

\title{The univalence axiom for elegant Reedy presheaves}
\author{Michael Shulman}
\email{shulman@sandiego.edu}
\address{Department of Mathematics and Computer Science,
  University of San Diego,
  5998 Alcal\'a Park,
  San Diego, CA, 92110,
  USA}
\thanks{This material is based upon work supported by the National Science Foundation under a postdoctoral fellowship and agreement No. DMS-1128155.  Any opinions, findings, and conclusions or recommendations expressed in this material are those of the author and do not necessarily reflect the views of the National Science Foundation.}


\keywords{homotopy type theory, univalence axiom, elegant Reedy category}

\begin{abstract}
  We show that Voevodsky's univalence axiom for homotopy type theory is valid in categories of simplicial presheaves on elegant Reedy categories.
  In addition to diagrams on inverse categories, as considered in previous work of the author, this includes bisimplicial sets and $\Theta_n$-spaces.
  This has potential applications to the study of homotopical models for higher categories.
\end{abstract}


\maketitle

\section{Introduction}
\label{sec:introduction}

\emph{Type theory} is a formal syntax for mathematical reasoning, with roots in constructive logic and computer science.
Its \emph{types} are traditionally viewed as set-like, but can more generally be interpreted as objects of any sufficiently structured category~\cite{seely:lccc-tt,cd:lccc-tt}; thus formal derivations in type theory yield theorems in category theory.
Recently it has emerged (see e.g.~\cite{hs:gpd-typethy,warren:thesis,aw:htpy-idtype,ak:htmtt,gb:topsimpid,voevodsky:typesystems,lw:localuniv,awodey:natmodels}) that this correspondence can be extended to certain \emph{model} categories, so that formal derivations in type theory can also yield theorems in homotopy theory.
The resulting subject is known as \emph{homotopy type theory}.

The collection of homotopical theorems that have been proven using type theory is small but growing; see~\cite[Chapter 8]{hottbook} for a list as of its publication.
So far, all such theorems were already known by methods of classical homotopy theory, but there are several advantages to the type-theoretic approach.
One significant one is that type theory can be interpreted in \emph{many} model categories, so that a type-theoretic proof of a theorem such as $\pi_1(S^1)=\mathbb{Z}$~\cite{ls:pi1s1} is much more general than a classical proof using topological spaces or simplicial sets.
For example, Lumsdaine, Finster, and Licata have already used type theory to produce a new proof of the Blakers--Massey theorem~\cite{favonia:blakers-massey-agda}, which thus applies in more general categories.
When translated across the categorical interpretation of type theory~\cite{rezk:hott-blakersmassey}, this proof becomes a model-categorical one, which could in principle have been discovered by classical homotopy theorists; but in practice it was not.

Type theory also brings a more ``internal'' perspective to homotopy theory, yielding more general constructions of various objects.
For instance, in classical homotopy theory one may consider the space $\mathrm{hAut}(X)$ of self-homotopy equivalences of a space $X$, defined simply as the obvious subspace of the function space $X^X$.
Homotopy type theory shows that an equivalent space to $\mathrm{hAut}(X)$ can be obtained by purely categorical constructions, which is thereby applicable in more generality.
(We will describe this example in more detail in \autoref{sec:equiv}.)

Finally, another advantage of type-theoretic homotopy theory is its convenient treatment of fibrations.
In type theory, a fibration over a space $A$ is represented by a \emph{type family} indexed by $A$, which is a map from $A$ into the \emph{universe type} $U$.
The correctness of this representation is guaranteed by Voevodsky's \emph{univalence axiom}, which identifies the path space of $U$ with a certain space of equivalences, implying that $U$ is a ``classifying space for fibrations''.
Of course, classifying spaces also exist in classical homotopy theory, but the systematic representation of fibrations as functions in type theory simplifies many things, especially when working with spaces defined as colimits (for example, the ``encode-decode method'' described in~\cite[Chapter 8]{hottbook}), or when doing parametrized homotopy theory.

With these two advantages of type theory in mind, it becomes important to know in \emph{which} model categories we can interpret type theory, and in particular the univalence axiom.
Since univalence says that the universe is a classifying space for fibrations, one natural categorical analogue would be the \emph{object classifiers} of Rezk and Lurie (see~\cite[\S6.1.6]{lurie:higher-topoi} and also~\cite{gk:univlcc}).
Thus, we may expect that type theory with the univalence axiom could be interpreted in any model category with object classifiers, and in particular in any presentation of an ``\io-topos''.

The main problem with this, and with the interpretation of type theory more generally, is that its formal syntax is stricter than the categorical structure available in the desired models.
The first model of the univalence axiom to overcome this difficulty was also due to Voevodsky~\cite{klv:ssetmodel}, using the model category \sSet of simplicial sets, which presents the \io-category of $\infty$-groupoids (the most basic \io-topos).

In~\cite{shulman:invdia}, starting from Voevodsky's model in \sSet, I constructed a model of type theory satisfying univalence in the \emph{Reedy model category} $\sSet^I$, whenever $I$ is an \emph{inverse category}.
This paper will generalize that result to the Reedy model structure on $\sSet^{\cC\op}$ whenever $\cC$ is an \emph{elegant Reedy category}, as in~\cite{br:reedy}.
(This result has now been further generalized by Cisinski~\cite{cisinski:elegant}.)
Elegant Reedy categories include direct categories (the opposites of inverse categories), but also categories such as the simplex category $\Delta$, the $n$-fold simplex category $\Delta^n$, and Joyal's categories $\Theta_n$.
Thus, the \io-toposes of $n$-fold simplicial spaces and $\Theta_n$-spaces admit models of type theory satisfying the univalence axiom.
We will not study such particular models further in this paper, but since these toposes have been used as models for higher categories (see e.g.~\cite{rezk:css,rezk:cpncats}), their internal type theories may be useful in extending the interpretation of type theory from \io-categories to $(\infty,n)$-categories.

The proof given in this paper does not depend on that of~\cite{shulman:invdia}, and is more similar in flavor to that of~\cite{klv:ssetmodel}.
In particular, it is purely model-category-theoretic; no knowledge of type theory is required.
(Note that the model-categorical parts of the construction of~\cite{klv:ssetmodel} are reproduced in the shorter paper~\cite{klv:univalence}.)
However, this new proof does not replace~\cite{shulman:invdia}, since it applies only to presheaves of simplicial sets, whereas that of~\cite{shulman:invdia} applies to diagrams in any category which models type theory with univalence (such as the syntactic category of type theory itself) and also generalizes to oplax limits and gluing constructions.

\begin{rmk}
  Most aspects of type theory aside from univalence (e.g.\ $\Sigma$-types, $\Pi$-types, and identity types) are now known to admit models in \emph{all} \io-toposes, and indeed in all locally presentable, locally cartesian closed \io-categories.
  By the coherence theorem of~\cite{lw:localuniv} (see also~\cite{awodey:natmodels}), it suffices to present such an \io-category by a ``type-theoretic model category'' in the sense of~\cite{shulman:invdia}, such as a right proper Cininski model category~\cite{cisinski:topos,cisinski:presheaves}.
  That this is always possible has been proven by Cisinski~\cite{cisinski:lccc-rpcmc} and by Gepner--Kock~\cite{gk:univlcc}.
  Moreover, if the \io-category is an \io-topos, then we can choose fibrations of fibrant objects representing its object classifiers, which will behave almost like universes and satisfy the univalence axiom.
  What is missing is that such ``universes'' need not be strictly closed under the type-forming operations; that is, the operation taking elements of the universe to types only respects these operations up to equivalence.
  It is this extra missing bit of strictness which we aim to provide here, in the special case of elegant Reedy presheaves.
\end{rmk}

In fact, a good deal of the proof that we will present is not special to elegant Reedy categories: it applies to any cofibrantly generated right proper model structure on a presheaf category whose cofibrations are the monomorphisms and such that the codomains of the generating acyclic cofibrations are representable.
The structure of the paper reflects this fact.
We begin in \S\ref{sec:on-univalence} with some remarks on how to construct univalent universes (that is, universe objects satisfying the univalence axiom).
In particular, we recall a method due to~\cite{klv:ssetmodel} which enables us to reduce statements about the universe, such as its univalence and fibrancy, to statements about fibrations.

In the next three sections \S\ref{sec:constr-univ}--\ref{sec:modeling-type-theory}, we show that any presheaf model category with the properties mentioned above satisfies almost all the requirements to represent a univalent universe in the internal type theory; the only thing missing is a proof that it is a fibrant object.
Specifically, in \S\ref{sec:constr-univ} we construct such universes in three ways (two of which are due to~\cite{klv:ssetmodel} and in~\cite{streicher:ttuniv}); in \S\ref{sec:equiv} we prove a postponed lemma from \S\ref{sec:on-univalence}; and in \S\ref{sec:modeling-type-theory} we verify the remaining requirements for modeling universes in type theory, using~\cite{lw:localuniv,klv:ssetmodel}.

Finally, in \S\ref{sec:elegant} we complete the proof in the case of elegant Reedy presheaves, showing that such presheaf categories satisfy the above conditions, and that moreover their universes are fibrant.
This depends on the explicit nature of the Reedy model structure.

A remark about notation: we will denote fibrations and cofibrations in any model category by $A\fib B$ and $A\cof B$, respectively.
Similarly, we write $A\afib B$ and $A\acof B$ for acyclic fibrations and acyclic cofibrations.

\subsection*{Acknowledgments}

I am grateful to Peter Lumsdaine for many very helpful discussions, and for emphasizing the essential outline of the proof in~\cite{klv:ssetmodel}, as described in \S\ref{sec:on-univalence}.
I am also grateful to Bas Spitters, and to the referee, for helpful comments and suggestions.

\section{On proofs of univalence}
\label{sec:on-univalence}

The univalence axiom, when interpreted in a model category, is a statement about a ``universe object'' $U$, which is fibrant and comes equipped with a fibration $p\colon \Util \fib U$ that is generic, in the sense that any fibration with ``small fibers'' is a pullback of $p$.
(The meaning of ``small'' will vary with the model; the important thing for modeling type theory is that the small fibrations be closed under the category-theoretic analogues of all the type forming operations.
Thus, we usually take the small fibrations to be those of cardinality smaller than some inaccessible cardinal, for some appropriate meaning of ``cardinality''.
We will return to this in \S\ref{sec:modeling-type-theory}.)

In homotopy theory, it would be natural to ask for the stronger property that $U$ is a \emph{classifying space} for small fibrations, i.e.\ that homotopy classes of maps $A\to U$ are in bijection with (rather than merely surjecting onto) equivalence classes of small fibrations over $A$.
The univalence axiom is a further strengthening of this: it says that the path space of $U$ is equivalent to the ``universal space of equivalences'' between fibers of $p$ (which we will define in \S\ref{sec:equiv}).
In particular, therefore, if two pullbacks of $p$ are equivalent, then their classifying maps are homotopic.

It is not difficult to obtain a fibrant univalent universe that classifies small fibrations up to homotopy, i.e.\ such that any fibration with small fibers is a \emph{homotopy} pullback of the generic one.
For instance, one can simply choose any fibration between fibrant objects that represents an object classifier in the sense of~\cite[\S6.1.6]{lurie:higher-topoi}.
However, for modeling type theory we are concerned with the class of fibrations occurring as \emph{strict} pullbacks of the generic one.
Finding a fibrant universe that classifies all fibrations with small fibers in this strict sense is where the difficulties lie in modeling the univalence axiom.

Naively, we might expect that the construction of such an object would take place in four steps:
\begin{enumerate}[label=(\arabic*),leftmargin=*]
\item Construct a particular small fibration $p\colon \Util \fib U$.\label{item:u1}
\item Prove that every small fibration is a (strict) pullback of $p$.\label{item:u2}
\item Prove that $U$ is fibrant.\label{item:u3}
\item Prove that the univalence axiom holds.\label{item:u4}
\end{enumerate}
The proof in~\cite{shulman:invdia} that univalence lifts to inverse diagrams follows this outline: we construct a Reedy fibration $p$ which satisfies~\ref{item:u2} and~\ref{item:u3} almost by definition, and then~\ref{item:u4} follows by a somewhat lengthy, but direct, analysis of exactly what the univalence axiom claims.

The proof of univalence for simplicial sets in~\cite{klv:ssetmodel}, by contrast, follows a slightly different route.
They first construct a fibration $p\colon \Util \fib U$ which satisfies the following stronger version of~\ref{item:u2}:
\begin{enumerate}[label=(\arabic*\/$'$),leftmargin=*,start=2]
\item Given the solid arrows in the following diagram, where $A\cof B$ is a cofibration, $P\fib B$ is a small fibration, and both squares of solid arrows are pullbacks:
  \begin{tikzcenter}[->,>=angle 60,scale=.8]
    \node (A) at (0,0) {$A$};
    \node (Q) at (0,2) {$Q$};
    \node (U) at (3,0) {$U$};
    \node (Util) at (3,2) {$\Util$};
    \node (P) at (1.5,1) {$P$};
    \node (B) at (1.5,-1) {$B$};
    \draw[>->] (A) -- (B); \draw[->>] (Q) -- (A); \draw (Q) -- (Util);
    \draw (A) -- (U); \draw[->>] (Util) -- (U); \draw (Q) -- (P);
    \draw [white,line width=5pt] (P) -- (B); \draw[->>] (P) -- (B);
    \draw[dashed] (B) -- (U); \draw[dashed] (P) -- (Util);
  \end{tikzcenter}
  there exist the dashed arrows rendering the diagram commutative and the third square also a pullback.\label{item:u2p}
\end{enumerate}
In a context (such as simplicial sets) where all objects are cofibrant, taking $A=\emptyset$ in~\ref{item:u2p} yields~\ref{item:u2}.

Condition~\ref{item:u2p} can be rephrased in the following suggestive way: suppose for the sake of argument that there were a thing called \fU such that maps $A\to\fU$ were \emph{precisely} small fibrations over $A$.
Then the small fibration $p\colon \Util\fib U$ would be classified by a map $U\to \fU$, and~\ref{item:u2p} asserts that this map is an acyclic fibration.
(One can even make this precise by regarding \fU as a fibered category or groupoid.)

With~\ref{item:u2p} in hand,~\ref{item:u3} and~\ref{item:u4} can be reduced to statements not referring to $U$ at all.
For instance, suppose we can show:
\begin{enumerate}[label=(\arabic*\/$'$),leftmargin=*,start=3]
\item If $i\colon A\cof B$ is an acyclic cofibration and $P\fib A$ a small fibration, then there exists a small fibration $Q\fib B$ such that $P\cong i^* Q$.\label{item:u3p}
  In other words, the solid arrows below can be completed to a pullback square as shown:
  \begin{tikzcenter}[scale=1.3]
    \node (A) at (0,0) {$A$};
    \node (B) at (1,0) {$B$};
    \node (P) at (0,1) {$P$};
    \node (Q) at (1,1) {$Q$};
    \draw[>->] (A) -- node[above] {\scriptsize $i$} node[below] {\scriptsize $\sim$} (B);
    \draw[->>] (P) -- (A);
    \draw[->>,dashed] (Q) -- (B);
    \draw[dashed,->] (P) -- (Q);
  \end{tikzcenter}
\end{enumerate}
Then given an acyclic cofibration $i\colon A\cof B$ and a map $f\colon A\to U$,~\ref{item:u3p} gives us a fibration $Q$ over $B$ which pulls back to $f^*\Util$ over $A$, and by~\ref{item:u2p} we have $g\colon B\to U$ with $g^*\Util\cong Q$ and $g i = f$.
Thus,~\ref{item:u3} follows.
Intuitively, we are saying that since $U\to\fU$ is an acyclic fibration, if \fU is fibrant then so is $U$.

For~\ref{item:u4}, we need to know the category-theoretic expression of the univalence axiom.
This asserts that a canonical map $P U \to \nEq(\Util)$ is an equivalence, where $PU$ denotes the path object of $U$ and $\nEq(\Util)$ is the universal space of equivalences over $U\times U$;
we will define the latter precisely in \S\ref{sec:equiv}.
By the 2-out-of-3 property, this is equivalent to $U \to \nEq(\Util)$ being an equivalence, and therefore also to either projection $\nEq(\Util)\to U$ being an equivalence.
Since these projections are always fibrations, we want them to be acyclic fibrations; but acyclic fibrations are characterized by a lifting property.
If we rephrase this property of the second projection in terms of actual fibrations and equivalences (i.e.\ using the hypothetical \fU), we obtain:
\def\Ebar{D}
\def\wbar{v}
\begin{enumerate}[label=(\arabic*\/$'$),leftmargin=*,start=4]
\item Suppose given a cofibration $i\colon A\cof B$, small fibrations $\Ebar_2 \fib B$ and $E_1 \fib A$, and an equivalence $w\colon E_1 \toiso E_2$ of fibrations over $A$, where $E_2 \coloneqq i^* \Ebar_2$.
  Then there exists a small fibration $\Ebar_1$ over $B$ and an equivalence $\wbar\colon \Ebar_1 \toiso \Ebar_2$ over $B$, which yields $w$ when pulled back along $i$.\label{item:u4p}
  \begin{tikzcenter}[->,>=angle 60,scale=1.8]
    \node (A) at (0,0) {$A$};
    \node (B) at (1.5,0) {$B$};
    \node (E1) at (-.3,1.3) {$E_1$};
    \node (E2) at (.3,.7) {$E_2$};
    \draw[->>] (E1) -- (A);
    \draw[->>] (E2) -- (A);
    \draw[>->] (A) -- node[auto,swap] {\scriptsize $i$} (B);
    \draw (E1) -- node[fill=white,inner sep=1pt] {\scriptsize $w$} (E2);
    \node (E2b) at (1.8,.7) {$\Ebar_2$};
    \draw[->>] (E2b) -- (B);
    \draw (E2) -- (E2b);
    \draw[-] (.5,.55) -- (.45,.45) -- (.35,.45);
    \begin{scope}[dashed]
      \node (E1b) at (1.2,1.3) {$\Ebar_1$};
      \draw (E1) -- (E1b);
      \draw (E1b) -- node[fill=white,inner sep=1pt,auto] {\scriptsize $\wbar$} (E2b);
      \draw[->>] (E1b) -- (B);
      \draw[-,solid] (0.22,1.1) -- (.25,1) -- (.15,1);
    \end{scope}
  \end{tikzcenter}
\end{enumerate}
If~\ref{item:u4p} holds, then for any commutative square
\begin{tikzcenter}[->,>=angle 60,scale=1.5]
  \node (A) at (0,1) {$A$};
  \node (B) at (0,0) {$B$};
  \node (U) at (1,0) {$U$};
  \node (Eq) at (1,1) {$\nEq(\Util)$};
  \draw (A) -- (Eq); \draw[->>] (Eq) -- (U); \draw[>->] (A) -- node[auto,swap] {\scriptsize $i$} (B); \draw (B) -- (U);
\end{tikzcenter}
with $i$ a cofibration, the given maps $A\to \nEq(\Util)$ and $B\to U$ respectively classify $w$ and $\Ebar_2$ as in~\ref{item:u4p}.
Then~\ref{item:u4p} gives $\Ebar_1$ and $\wbar$, condition~\ref{item:u2p} yields a classifying map for $\Ebar_1$ extending the composite $A\to \Util\to U$, and using the following lemma, we can construct a lift in the above square, so that $\nEq(\Util)\to U$ is an acyclic fibration.

\begin{lem}\label{thm:eqlift}
  In a suitable model category, let $\Ebar_1\fib B$ and $\Ebar_2\fib B$ be fibrations classified by maps $B\toto U$, let $\wbar\colon \Ebar_1\to \Ebar_2$ be a weak equivalence over $B$, let $i\colon A\cof B$ be a cofibration, and suppose we have a lift of $A \xto{i} B \to U\times U$ to $\nEq(\Util)$ which classifies $i^*(\wbar)$.
  Then this lift can be extended to $B$ so as to classify $\wbar$.
\end{lem}

In particular, if all objects are cofibrant (as will be the case in all our examples), then taking $A=\emptyset$ in \autoref{thm:eqlift} implies that any weak equivalence between fibrations over $B$ is classified by some map $B\to \nEq(\Util)$.

The proof of \autoref{thm:eqlift} is the only place where we need to know the actual definition of $\nEq(\Util)$.
This definition is determined by the specific formulation of the univalence axiom in type theory and is somewhat technical, so we will consider it separately in \S\ref{sec:equiv}, postponing the proof of \autoref{thm:eqlift} until then.
For now, it suffices to take \autoref{thm:eqlift} as a (hopefully plausible) black box.
In fact, one might argue that just as~\ref{item:u2p} determines a good notion of what it \emph{means} to be a universe object in a model category, the conclusion of \autoref{thm:eqlift} is a good definition of what it \emph{means} for $\nEq(\Util)$ to be a ``universal space of equivalences'' therein.

\section{Constructing univalent universes}
\label{sec:constr-univ}

Let us now consider in what level of generality we can prove~\ref{item:u2p},~\ref{item:u3p}, and~\ref{item:u4p}.
Perhaps surprisingly, given that~\ref{item:u4p} is a modification of the actual statement~\ref{item:u4} of univalence, it seems to be the easiest to prove in the most generality.
The proof in~\cite{klv:ssetmodel} for simplicial sets carries through almost word-for-word in a much more general context.

Let $\cC$ be a small category.
We say a morphism $f\colon A\to B$ in the presheaf category $\bSet^{\cC\op}$ is \textbf{\ka-small}, for some cardinal number \ka, if for all $c\in \cC$ and $b\in B_c$ we have $|f_c^{-1}(b)|<\ka$.
We denote by $|\cC|$ the cardinality of the coproduct set $\sum_{c,c'\in\cC} \cC(c,c')$ consisting of all arrows in $\cC$.

\begin{thm}\label{thm:4p}
  If $\bSet^{\cC\op}$ is a presheaf category that is a simplicial model category whose cofibrations are exactly the monomorphisms, and $\ka$ is a cardinal number larger than $|\cC|$, then the \ka-small fibrations in $\bSet^{\cC\op}$ satisfy~\ref{item:u4p}.
\end{thm}
\begin{proof}[Proof (from~\cite{klv:ssetmodel})]
  Suppose given a cofibration $i\colon A\cof B$, a \ka-small fibration $\Ebar_2 \fib B$, and an equivalence $w\colon E_1 \toiso E_2 \coloneqq i^* \Ebar_2$ of \ka-small fibrations over $A$; 
  we want to construct $\Ebar_1$ and the dashed arrows shown in the diagram shown below~\ref{item:u4p}.
  Define $\Ebar_1$ and $\wbar$ as the following pullback in $\bSet^{\cC\op}/B$, where $i_*$ denotes the right adjoint of pullback $i^*$:
  \[\vcenter{\xymatrix{
      \Ebar_1\ar[r]\ar[d]_{\wbar} \pullbackcorner &
      i_* E_1\ar[d]^{i_*(w)}\\
      \Ebar_2\ar[r]_-\eta &
      i_* i^* \Ebar_2 \mathrlap{\;\cong i_* E_2}.
    }}\]
  Since $i^*$ preserves this pullback, and $i_*$ is fully faithful, 
  $\wbar$ pulls back to $w$.
  It is straightforward to check that $\Ebar_1\to B$ is \ka-small; it remains to show it is a fibration and that $\wbar$ is an equivalence.

  We factor $w$ as an acyclic cofibration followed by an acyclic fibration and treat the two cases separately.
  In the second case, 
  $i_*(w)$ is an acyclic fibration and thus so is $\wbar$. 
  In the first case, since $E_1$ and $E_2$ are fibrations over $A$, by~\cite[7.6.11 and 9.5.24]{hirschhorn:modelcats}, we have a simplicial deformation retraction $H\colon \Delta^1 \otimes E_2 \to E_2$ of $E_2$ onto $E_1$ in $\bSet^{\cC\op}/A$, where $\otimes$ denotes the tensor for the simplicial enrichment.
  Now $\eta$ and $\wbar$ are monic, so
  if $P$ denotes the pushout
  \[\vcenter{\xymatrix{
      E_1 \ar[r]\ar[d]_w &
      \Ebar_1\ar[d] \ar[ddr]^{\wbar}\\
      E_2 \ar[r] \ar[drr] &
      P \ar[dr]|(.3)j \pushoutcorner \\
      && \Ebar_2,
    }}\]
  then $j$ is also a monomorphism.
  Since we are in a simplicial model category, the pushout product on the left of the following square is an acyclic cofibration:
  \[\vcenter{\xymatrix{
      (\Delta^0 \otimes \Ebar_2) \cup_{\Delta^0 \otimes P} (\Delta^1 \otimes P)\ar[r]\ar[d] &
      \Ebar_2\ar[d]\\
      \Delta^1 \otimes \Ebar_2 \ar[r] \ar@{.>}[ur]_(.6)\Hbar &
      B.
    }}\]
  The map at the top is induced by the identity on $\Delta^0 \otimes \Ebar_2 \cong \Ebar_2$, and on $\Delta^1 \otimes P$ by a combination of $\eta H$ on $E_2$ and the constant homotopy at $\wbar$ on $\Ebar_1$ (which agree in $E_1$ since $H$ is a deformation retraction).
  Since $\Ebar_2\to B$ is a fibration, \Hbar exists, and since $i^*(\Hbar)=H$ is a deformation retraction into $E_1$, $\Hbar$ is a deformation retraction into $\Ebar_1$.
  Thus $\wbar$ is the inclusion of a deformation retract, hence a weak equivalence; and $\Ebar_1\to B$, being a retract of $\Ebar_2\to B$, is a fibration.
\end{proof}

I do not know any general context of this sort in which one can prove~\ref{item:u3p}.
However, the situation with~\ref{item:u2p} is better:

\begin{thm}\label{thm:wellordered}
  Suppose $\bSet^{\cC\op}$ is a presheaf category that is a cofibrantly generated model category in which all cofibrations are monomorphisms, and that the codomains of the generating acyclic cofibrations are representable.
  Then there exists a \ka-small fibration satisfying~\ref{item:u2p}.
\end{thm}

In the special case of simplicial sets, this theorem has been proven by~\cite{klv:ssetmodel} and~\cite{streicher:ttuniv}, and their proofs generalize immediately to any category satisfying the stated hypotheses.
I will present a third, new, proof of this theorem, which I believe makes the connection to $(\infty,1)$-categorical object classifiers rather clearer.
But to facilitate comparisons, I will first sketch the main ideas of the proofs of~\cite{klv:ssetmodel} and~\cite{streicher:ttuniv}.

The basic idea of both of these proofs is that a presheaf $U$ is (of course) defined by giving its values at each object $c\in \cC$, while by the Yoneda lemma, elements of $U(c)$ are in bijective correspondence with maps $Y_c \to U$, where $Y_c$ is the representable presheaf at $c$.
Since maps into $U$ are supposed to classify small fibrations, we should define $U(c)$ to be some set of small fibrations over $Y_c$.
The problem is to choose a small set of representatives for this collection of fibrations in such a way that $U$ becomes a strict functor (rather than a pseudofunctor).
In~\cite{klv:ssetmodel} this is done by imposing well-orderings on the fibers, while in~\cite{streicher:ttuniv} it is done by considering presheaves on the category of elements of $Y_c$ (i.e.\ the slice category $\cC/c$).

By contrast, in the proof I will now present, we do not define $U$ by giving its value at each object.
Indeed, the fact that we are in a presheaf category makes no overt appearance; we will only need to know that the codomains of the generating acyclic cofibrations $X\acof Y$ have the property that $\hom(Y,-)$ preserves small colimits.
In addition, we will need the following ``exactness properties'' of any Grothendieck topos.
\begin{enumerate}[leftmargin=*,label=(\alph*)]
\item Given a family
  \[\vcenter{\xymatrix{
      X_i\ar[r]\ar[d] &
      Z\ar[d] \\
      A_i\ar[r] &
      \sum_i A_i
    }}\]
  of commutative squares in which the bottom family of morphisms are the injections into a coproduct $(A_i \to \sum_i A_i)_{i\in I}$ and the right-hand map is the same for all $i$, then the top family of morphisms form a coproduct diagram (so that $Z \cong \sum_i X_i$) if and only if all the squares are pullbacks.
  A category with this property is called (infinitary) \textbf{extensive}~\cite{clw:ext-dist}.
  Extensivity is equivalent to coproducts being stable and disjoint, and implies that coproducts preserve monomorphisms and pullback squares.
\item Given a commutative cube
  \[ \xymatrix@-.5pc{
    X \ar[rr] \ar[dd] \ar[dr] &&
    Z \ar'[d][dd] \ar[dr]\\
    & Y \ar[dd] \ar[rr] &&
    W \ar[dd]\\
    A \ar@{>->}[dr] \ar'[r][rr] & &
    C \ar[dr]\\
    & B \ar[rr] &&
    D
  }\]
  in which $A\cof B$ is a monomorphism, the bottom face is a pushout, and the left and back faces are pullbacks, then the top face is a pushout if and only if the front and right faces are pullbacks.
  A category with this property is called \textbf{adhesive}~\cite{ls:adhesive}.
  Adhesivity is equivalent to pushout squares of monomorphisms being also pullback squares and being stable under pullback~\cite{gl:adhesive}, and implies that the pushout of a monomorphism is a monomorphism.
\item Given a commutative diagram
  \[\vcenter{\xymatrix{
      X_0\ar[r]\ar[d] &
      X_1\ar[r]\ar[d] &
      \cdots\ar[r] &
      X_\al\ar[r]\ar[d] &
      \cdots\ar[r] &
      X_\la\ar[d]\\
      A_0\ar[r] &
      A_1\ar[r] &
      \cdots\ar[r] &
      A_\al\ar[r] &
      \cdots\ar[r] &
      A_\la
    }}\]
  of transfinite sequences for $\al <\la$, with \la some limit ordinal, in which the bottom row is a colimit diagram, and for each $\al<\be<\la$ the morphism $A_\al \to A_\be$ is a monomorphism and the square
  \[\vcenter{\xymatrix{
      X_\al\ar[r]\ar[d] &
      X_\be\ar[d]\\
      A_\al\ar[r] &
      A_\be
    }}\]
  is a pullback, then the top row is a colimit diagram if and only if for each $\al<\la$ the square
  \[\vcenter{\xymatrix{
      X_\al\ar[r]\ar[d] &
      X_\la\ar[d]\\
      A_\al\ar[r] &
      A_\la
    }}\]
  is a pullback.
  I have not been able to find a name in the literature for categories with this property; I propose to call them \textbf{exhaustive}.
  Exhaustivity is equivalent to asking that in a transfinite composite of monomorphisms, the coprojections into the colimit are also monomorphisms and the colimit is pullback-stable~\cite{nlab:exhaustive}.
  It implies that transfinite composites of monomorphisms preserve pullbacks, and hence also monomorphisms.
\end{enumerate}

\begin{proof}[
  Proof of \autoref{thm:wellordered}]
  Let \cI be a generating set of cofibrations.
  The proof may be described as ``constructing a cofibrant replacement of \fU by the small object argument,'' where \fU is the hypothetical object classifying small fibrations on the nose.
  We define a transfinite sequence
  \[\vcenter{\xymatrix{
      \Util_0\ar@{>->}[r]\ar@{->>}[d] &
      \Util_1\ar@{>->}[r]\ar@{->>}[d] &
      \Util_2\ar@{>->}[r]\ar@{->>}[d] &
      \dots\ar@{>->}[r] &
      \Util_\al\ar@{>->}[r]\ar@{->>}[d] &
      \dots\\
      U_0\ar@{>->}[r] &
      U_1\ar@{>->}[r] &
      U_2\ar@{>->}[r] &
      \dots \ar@{>->}[r] &
      U_\al\ar@{>->}[r] &
      \dots
    }}\]
  such that
  \begin{enumerate}[leftmargin=*]
  \item Each map $\Util_\al \to U_\al$ is a \ka-small fibration.\label{item:af1}
  \item For $\al<\be$, the map $U_\al \to U_\be$ is monic.\label{item:af3}
  \item For $\al<\be$, the square
    \[\vcenter{\xymatrix{
        \Util_\al\ar@{>->}[r]\ar@{->>}[d] &
        \Util_\be\ar@{->>}[d]\\
        U_\al\ar@{>->}[r] &
        U_\be
      }}\]
    is a pullback (hence $\Util_\al \to \Util_\be$ is also monic).\label{item:af2}
  \end{enumerate}
  For limit $\al$, we take colimits of both sequences.
  (Including $\al=0$ as a limit, this means we begin with $\Util_0 = U_0 = \emptyset$).
  By induction, these are colimits of monomorphisms and all intermediate squares are pullbacks.
  Thus, by exhaustivity,~\ref{item:af3} and~\ref{item:af2} remain true in the colimit.
  For~\ref{item:af1}, it suffices to show that every commutative square as on the left below has a lift, where $X\cof Y$ is a generating acyclic cofibration.
  \begin{equation}
    \vcenter{\xymatrix{
        X \ar[r]\ar[d] &
        \Util_\al\ar[d]\\
        Y\ar[r] &
        U_\al
      }}
    \quad=\quad
    \vcenter{\xymatrix{
        X \ar[r]\ar[d] &
        \Util_\gm \ar[d] \ar[r] \pullbackcorner &
        \Util_\al\ar[d]\\
        Y\ar[r] &
        U_\gm \ar[r] &
        U_\al.
      }}
  \end{equation}
  However, by assumption this means the object $Y$ is a representable presheaf, and thus the map $Y \to U_\al$ factors through $\Util_\gm$ for some $\gm<\al$.
  Since $\Util_\gm$ is the pullback of $\Util_\al$ to $U_\gm$, any commutative square as on the left above factors as on the right, and since $\Util_\gm \to U_\gm$ is a fibration we can find a lift.

  At a successor stage, given $\Util_\al \to U_\al$ we consider the set of pairs $(i,f,p)$, where
  \begin{itemize}[leftmargin=1.5em]
  \item $i\colon A\to B$ is in \cI;
  \item $f\colon A\to U_\al$ is any morphism; and
  \item $p\colon P\to B$ lies in a small set of representatives for isomorphism classes of \ka-small fibrations into $B$ which are equipped with an isomorphism $i^*P \cong f^*\Util_\al$.
  \end{itemize}
  We define $\Util_{\al+1} \to U_{\al+1}$ to make the top and bottom squares of the following cube into pushouts:
  \[ \xymatrix@-.5pc{
    \sum_{(i,f,p)} f^*\Util_\al \ar[rr] \ar[dd] \ar[dr] &&
    \Util_\al \ar'[d][dd] \ar[dr]\\
    & \sum_{(i,f,p)} P \ar[dd]^(.3){\sum p} \ar[rr] &&
    \Util_{\al+1} \ar[dd]\\
    \sum_{(i,f,p)} A \ar[dr]_{\sum i} \ar'[r]^-{[f]}[rr] & &
    U_\al \ar[dr]\\
    & \sum_{(i,f,p)} B \ar[rr] &&
    U_{\al+1}.
  }\]
  Note that the coproducts $\sum_{(i,f,p)}$ exist because the set of triples $(i,f,p)$ is small, for which purpose it is essential that $p$ belong to a small set of representatives for isomorphism classes.

  Now by extensivity, $\sum i$ is monic, so by adhesivity, so is $U_\al \to U_{\al+1}$, giving~\ref{item:af3}.
  Likewise, by extensivity, the left and back faces are pullbacks, so by adhesivity, so are the right and front faces, giving~\ref{item:af2}.

  Finally, for any generating acyclic cofibration $X\to Y$, since $Y$ is representable, any map $Y \to U_{\al+1}$ factors through $\sum_{(i,f,p)} B$ or $U_\al$.
  Since the front and right faces are pullbacks, it follows that any commutative square of the form
  \begin{equation}
    \vcenter{\xymatrix{
        X \ar[r]\ar[d] &
        \Util_{\al+1}\ar[d]\\
        Y\ar[r] &
        U_{\al+1}
      }}\label{eq:pushout-lift}
  \end{equation}
  factors through either $\sum p$ or $\Util_\al \to U_\al$, both of which are fibrations (the former by a similar argument using extensivity).
  Thus, we can lift in any square~\eqref{eq:pushout-lift}, so $\Util_{\al+1} \to U_{\al+1}$ is a \ka-small fibration; thus~\ref{item:af1} holds.

  Now since a presheaf category is locally presentable, there exists $\la$ such that the domains of all morphisms in \cI are $\la$-compact (i.e.\ their covariant representable functors preserve \la-filtered colimits).
  For such a $\la$, if $i\colon A\cof B$ is in \cI, and $f\colon A\to U_\la$ is a morphism, and $p\colon P\fib B$ is a \ka-small fibration with $i^*P \cong f^*\Util_\la$, then by \la-compactness of $A$, $f$ factors through $U_\al$ for some $\al<\la$.
  By construction, $(i,f,p)$ then induces a map $g\colon B\to U_{\al+1}$ with $g^*\Util_{\al+1} \cong P$; so~\ref{item:u2p} holds for $i\in\cI$.
  It will suffice, therefore, to prove the following lemma.
\end{proof}

\begin{lem}
  In a category satisfying the assumptions of \autoref{thm:wellordered}, let $p:\Util\to U$ be a fibration.
  Then the class of monomorphisms $i$ satisfying~\ref{item:u2p} with respect to $p$ is closed under pushout, transfinite composition, and retracts (i.e.\ it is ``saturated'').
\end{lem}
\begin{proof}
  For closure under pushouts, suppose given the solid arrows in the following diagram:
  
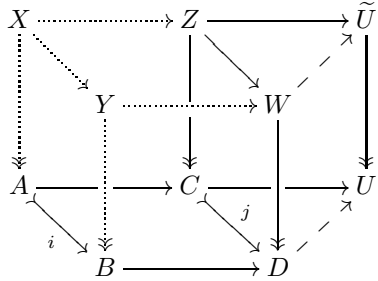
\[ \xymatrix@-.5pc{
    X \ar@{.>}[rr] \ar@{.>>}[dd] \ar@{.>}[dr] &&
    Z \ar@{->>}'[d][dd] \ar[dr] \ar[rr] &&
    \Util \ar@{->>}[dd]\\
    & Y \ar@{.>>}[dd] \ar@{.>}[rr] &&
    W \ar@{->>}[dd] \ar@{-->}[ur] \\
    A \ar@{>->}[dr]_i \ar'[r][rr] & &
    C \ar@{>->}[dr]^j \ar'[r][rr] && U\\
    & B \ar[rr] &&
    D\ar@{-->}[ur]
  }\]
  where the bottom square is a pushout, $i$ (hence also $j$) is monic, and the other two squares of solid arrows are pullbacks.
  Then we can fill in the objects $X$ and $Y$ and the dotted arrows to make all vertical faces of the cube pullbacks; hence by adhesivity the top face is a pushout.
  Assuming $i$ satisfies~\ref{item:u2p}, we have a map $B\to U$ which pulls back $\Util$ to $Y$ compatibly; thus the universal property of pushouts induces the dashed arrows shown.
  Finally, stability of pushouts under pullback implies that the square involving the dashed arrows is also a pullback.

  For closure under transfinite composites, suppose $A_0 \cof A_\la$ is a transfinite composite of monomorphisms, and suppose given the solid arrows in the following diagram making the left-hand rhombus and the outer rectangle pullbacks:
  \[\vcenter
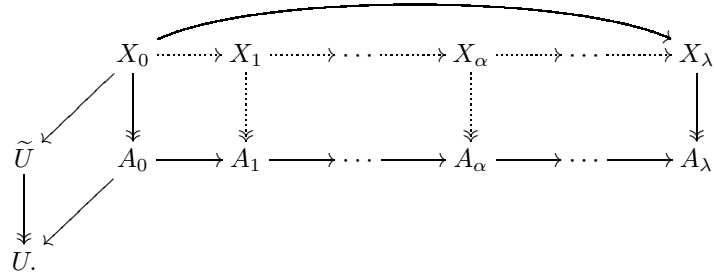
{\xymatrix{
      \\
      &X_0\ar@{.>}[r]\ar@{->>}[d] \ar[dl] \ar@(ur,ul)[rrrrr] &
      X_1\ar@{.>}[r]\ar@{.>>}[d] &
      \cdots\ar@{.>}[r] &
      X_\al\ar@{.>}[r]\ar@{.>>}[d] &
      \cdots\ar@{.>}[r] &
      X_\la\ar@{->>}[d]\\
      \Util\ar@{->>}[d] &A_0\ar[r] \ar[dl] &
      A_1\ar[r] &
      \cdots\ar[r] &
      A_\al\ar[r] &
      \cdots\ar[r] &
      A_\la\\
      U.
    }}\]
  Then we can fill in the $X_\al$ and the dotted arrows making the other squares all pullbacks; hence by exhaustivity the top row is a colimit.
  Assuming each $A_\al \cof A_\be$ satisfies~\ref{item:u2p}, we can successively extend the maps $A_0 \to U$ and $X_0 \to \Util$ to all $A_\al$ and $X_\al$, and hence in the colimit to $A_\la$ and $X_\la$.
  Finally, stability of transfinite composites under pullback implies that the induced squares are all also pullbacks.

  For closure under retracts, suppose given the following solid arrows:
  \[ \xymatrix@-.5pc{
    Z \ar@{.>}[rr] \ar@{->>}[dd] \ar[dr] &&
    X \ar@{.>}[rr] \ar@{.>>}[dd] \ar@{.>}[dr] &&
    Z \ar@{->>}'[d][dd] \ar[dr] \ar[rr] &&
    \Util \ar@{->>}[dd]\\
    & W\ar[dd] \ar@{.>}[rr] &&
    Y \ar@{.>>}[dd] \ar@{.>}[rr] &&
    W \ar@{->>}[dd] \ar@{-->}[ur] \\
    C \ar@{>->}[dr]_j \ar'[r][rr] &&
    A \ar@{>->}[dr]_i \ar'[r][rr] &&
    C \ar@{>->}[dr]^j \ar'[r][rr] && U\\
    & D \ar[rr] &&
    B \ar[rr] &&
    D\ar@{-->}[ur]
  }\]
  where the composites $C\to A\to C$ and $D\to B\to D$ are identities.
  Then we can fill in $X$ and $Y$ and the dotted arrows making all squares pullbacks.
  Assuming $i$ satisfies~\ref{item:u2p}, we have a map $B\to U$ compatibly classifying $Y$, and then the composite $D\to B\to U$ compatibly classifies $W$.
\end{proof}

I say that this proof makes the connection to object classifiers clearer because it depends mainly on the fact that the pseudo 2-functor
\begin{align*}
  \left(\bSet^{\cC\op}\right)\op &\to \bCat\\
  B &\mapsto \{ \text{fibrations over } B \}
\end{align*}
preserves coproducts, pushouts of monomorphisms, and transfinite composites of monomorphisms.
These can of course be regarded as ``stack'' conditions.
Moreover, since the monomorphisms in question are cofibrations, these colimits are also \emph{homotopy} colimits.

By comparison, in~\cite[6.1.6.3]{lurie:higher-topoi} object classifiers in an $(\infty,1)$-category \sC are constructed under the assumption that the $(\infty,1)$-functor
\begin{align*}
  \sC\op &\to (\infty,1)\bCat\\
  B &\mapsto \{ \text{all morphisms into } B \}
\end{align*}
preserves \emph{all} (homotopy) colimits.
In this situation one can simply apply the $(\infty,1)$-categorical adjoint functor theorem, but this could be unraveled more explicitly into a transfinite construction very like that in the third proof of \autoref{thm:wellordered}.

In conclusion, we have a general context which is \emph{almost} enough to construct fibrant univalent universes; in any particular situation it suffices to check the one remaining condition~\ref{item:u3p}.
We will do this for elegant simplicial presheaves in \S\ref{sec:elegant}; but first we have some holes to fill in.

\section{Universal spaces of equivalences}
\label{sec:equiv}

In this section we will define the universal space of equivalences $\nEq(\Util)$ and prove \autoref{thm:eqlift}.
The definition is exactly the categorical interpretation of Voevodsky's type-theoretic definition of equivalences.
However, in keeping with the tone of this paper, we will describe it without reference to type theory.

For all of this section, let \sM be a locally cartesian closed, right proper, simplicial model category whose cofibrations are the monomorphisms.
This is roughly what is needed for it to interpret type theory (it is somewhat stronger than being a \emph{type-theoretic model category} in the sense of~\cite{shulman:invdia}).
In particular, all objects of \sM are cofibrant.

Local cartesian closure implies that for any $f:A\to B$, the pullback functor $f^*:\sM/B \to \sM/A$ has a right adjoint, which we denote $\Pi_f$.
By adjointness, since $f^*$ preserves cofibrations (i.e.\ monomorphisms), $\Pi_f$ preserves acyclic fibrations.
Of course, $f^*$ also has a left adjoint given by composing with $f$, which we denote $\Sigma_f$.
If $f$ is a fibration, then $\Sigma_f$ maps every fibration $p:E\fib A$ to a fibration $fp:E\fib B$.
We write $\Pi_A$ and $\Sigma_A$ when $f$ is the map $A\to 1$ to the terminal object.

In general, the goal of the constructions we will present is to ``internalize'' statements like ``$f$ and $g$ are homotopic'' or ``$f$ is an equivalence''.
By this we mean, to first approximation, that we construct an object $V$ of $\sM$ such that there is a map $1\to V$ if and only if the statement in question holds.
(More precisely, we want there to be a map $X\to V$ if and only if the statement holds after pulling back to $\sM/X$.)
We may think of $V$ as a ``space'' whose points are, up to homotopy, \emph{assertions} or \emph{witnesses} of the statement in question --- but we construct it abstractly, using category-theoretic operations, rather than any knowledge we have about how the objects of $\sM$ are put together.

For example, given suppose given two maps $f,g:A\to B$ between fibrant objects, and suppose we would like to internalize the statement ``$f$ is homotopic to $g$''.
Externally (i.e.\ as a statement about \sM), to say that $f$ is homotopic to $g$ is to say that $(f,g):A\to B\times B$ lifts to a path object $P B$ for $B$.
Equivalently, this means that the pullback fibration $(f,g)^* P B \fib A$ has a section.
Finally, by adjointness, to give a section of this fibration is equivalent to giving a map $1 \to \Pi_A (f,g)^* (P B)$.
Thus, $\Pi_A (f,g)^* (P B)$ is a good choice for an internalization of ``$f$ is homotopic to $g$''.
(One can check that, in fact, there is a map $X\to \Pi_A (f,g)^* (P B)$ if and only if $X\times f$ is homotopic to $X\times g$ in $\sM/X$.)

Our primary interest in this section is in internalizing the statement ``$f$ is an equivalence''.
There are now many known ways to do this; we use the original one due to Voevodsky.
Let $p:E\fib B$ be a fibration between fibrant objects; we begin by internalizing the statement ``$p$ is an acyclic fibration.''
By~\cite[7.6.11(2)]{hirschhorn:modelcats}, $p$ is an acyclic fibration if and only if there is a section $s:B\to E$ (so that $p s = 1_B$) and a fiberwise homotopy $s p \sim 1_E$ (i.e.\ a homotopy in $\sM/B$).

Let $P_B E = (E\fib B)^{\Delta^1}$ be the cotensor in $\sM/B$ of $E\fib B$ by the standard interval $\Delta^1$.
Since $E$ is fibrant in $\sM/B$, $P_B E$ is a valid path object for it, i.e.\ we have an acyclic cofibration $E \acof P_B E$ and a fibration $P_B E \fib E\times_B E$ factoring the diagonal $E\to E\times_B E$.
Thus, given $s:B\to E$ with $p s = 1_B$, a fiberwise homotopy $s p \sim 1_E$ means a lift of $(sp,1_E) : E \to E\times_B E$ to $P_B E$.
Now we have a pullback square
\begin{equation}
  \vcenter{\xymatrix@C=3pc{
      E\ar[r]^-{(p s, 1_E)}\ar[d]_p \pullbackcorner &
      E\times_B E\ar[d]^{\pi_2}\\
      B\ar[r]_s &
      E
    }}
\end{equation}
where $\pi_2:E\times_B E \to E$ denotes the second projection.
Therefore, a lift of $(sp,1_E)$ to $P_B E$ is equivalent to a lift of $s:B\to E$ to $\Pi_{\pi_2} (P_B E)$.
Since $s$ is a section of $p$, to give $s$ and the homotopy $s p \sim 1_E$ together is equivalent to giving a section of the composite $\Pi_{\pi_2} (P_B E) \to E \xfib{p} B$, i.e.\ of $\Sigma_p \Pi_{\pi_2} (P_B E)$.
This motivates us to define
\begin{equation}
  \iscontr_B(E) \coloneqq \Sigma_p \Pi_{\pi_2} (P_B E).
\end{equation}
Note that $\iscontr_B(E)$ is an object of $\sM/B$.
We regard it as a $B$-indexed family of spaces internalizing, for each $b\in B$, the assertion that the fiber $p^{-1}(b)$ is contractible.

\begin{lem}\label{thm:iscontr}
  For a fibration $p:E\fib B$, the following are equivalent:
  \begin{enumerate}
  \item $p$ is an acyclic fibration.\label{item:ic1}
  \item $\iscontr_B(E) \fib B$ has a section.\label{item:ic2}
  \item There is a map $1 \to \Pi_B \iscontr_B(E)$.\label{item:ic2a}
  \item $\iscontr_B(E) \fib B$ is an acyclic fibration.\label{item:ic3}
  \end{enumerate}
\end{lem}
\begin{proof}
  We have already argued that~\ref{item:ic1}$\Leftrightarrow$\ref{item:ic2}, and \ref{item:ic2}$\Leftrightarrow$\ref{item:ic2a} is immediate from the adjunction defining $\Pi_B$.
  And certainly~\ref{item:ic3}$\Rightarrow$\ref{item:ic2}, so it will suffice to show~\ref{item:ic1}$\Rightarrow$\ref{item:ic3}.
  However, if $p$ is an acyclic fibration, then both projections $E\times_B E \to E$ are weak equivalences.
  Thus, by the 2-out-of-3 property, so is the diagonal $E\to E\times_B E$.
  Again by the 2-out-of-3 property, therefore, $P_B E \fib E\times_B E$ is an acyclic fibration.
  But $\Pi_{\pi_2}:\sM/(E\times_B E) \to \sM/E$ preserves acyclic fibrations, so $\iscontr_B(E) \fib B$ is the composite of two acyclic fibrations.
\end{proof}

The equivalence with \autoref{thm:iscontr}\ref{item:ic3} means that, informally, the ``points'' of $\iscontr_B(E)$ are ``no more than'' assertions that the corresponding fiber of $p$ is contractible.
In other words, a fibration can ``be acyclic'' in at most one way, up to homotopy.
We also have the following stronger property alluded to above:

\begin{lem}\label{thm:iscontr2}
  For a fibration $p:E\fib B$ and a map $g:A\to B$, the following are equivalent:
  \begin{enumerate}
  \item The pullback $g^*(E) \to A$ is an acyclic fibration.\label{item:ic2-1}
  \item $g:A\to B$ lifts to $\iscontr_B(E)$.\label{item:ic2-2}
  \end{enumerate}
\end{lem}
\begin{proof}
  First note that~\ref{item:ic2-2} is equivalent to saying that $g^* \iscontr_B(E)$ has a section.
  However, the construction of $\iscontr_B(E)$ involves only operations that are stable (up to isomorphism) under pullback along $g$ (in the case of $\Sigma$ and $\Pi$ this is sometimes called the ``Beck--Chevalley property'').
  Thus, $g^* \iscontr_B(E) \cong \iscontr_{A}(g^* E)$, so the equivalence follows from \autoref{thm:iscontr}.
\end{proof}

Now, in order to internalize the statement ``$f$ is an equivalence'', we will use the fact that $f$ is an equivalence if and only if the fibration-replacement of $f$ being an acyclic fibration.
Internally, this means we will assert that every \emph{homotopy} fiber of $f$ is contractible.

Specifically, given a map $f:E_1 \to E_2$, we define $P f$ as the pullback shown below.
\[\vcenter{\xymatrix{
    E_1 \ar[r]^f \ar@{.>}[dr]^r \ar[ddr]_{1_{E_1}\times f} & E_2 \ar[dr]\\
    &P f\ar[r]\ar@{->>}[d] \pullbackcorner &
    P E_2\ar@{->>}[d]\\
    &E_1\times E_2\ar[r]_{f\times 1_{E_2}}&
    E_2\times E_2.
  }}\]
This is a version of the classical mapping path fibration.
It has a universal property saying that to give a map $A\to P f$ is the same as to give a map $x:A\to E_1$, a map $y:A\to E_2$, and a simplicial homotopy $f x \sim y$.
The induced map $r:E_1 \to P f$ corresponds to $x=1_{E_1}$, $y = f$, and the constant homotopy.

We denote the two composites $P_B f \fib E_1\times E_2 \to E_1$ and $P_B f\fib E_1\times E_2 \to E_2$ by $q$ and $p$ respectively.
Interpreted representably, they remember only the maps $x$ and $y$ respectively.
By construction, we have $q r = 1_{E_1}$ and $p r = f$.

The composite $r q : P_B f \to P_B f$ acts representably by taking $x$, $y$, and a homotopy $H:f x \sim y$ to the triple consisting of $x$, $f x$, and the constant homotopy.
This map is homotopic to the identity, so $r$ admits a deformation retraction, hence is a weak equivalence.
(In fact, since $r$ admits the retraction $q$, it is monic and hence an acyclic cofibration; this is proven in~\cite{shulman:invdia} in a bit more generality, by mimicking the type-theoretic proof in~\cite{gg:idtypewfs}.)

In conclusion, $f = p r$ factors $f$ as a weak equivalence followed by a fibration.
By the 2-out-of-3 property, therefore, $f$ is a weak equivalence if and only if $p$ is an acyclic fibration, and therefore if and only if $\iscontr_{E_2}(P f) \fib E_2$ has a section (in which case it is also acyclic, by \autoref{thm:iscontr}).
Thus, if we define
\[ \isequiv(f) \coloneqq \Pi_{E_2} \iscontr_{E_2}(P f). \]
then there is a map $1\to \isequiv (f)$ if and only if $f$ is a weak equivalence.

In fact, we need a more general version of this construction that works with a fiberwise map between fibrations.
For two fibrations $p_1:E_1 \fib B$ and $p_2:E_2\fib B$ and a map $f:E_1 \to E_2$ a map over $B$, we define $P_B f$ by the analogous pullback:
\[\vcenter{\xymatrix{
    E_1 \ar[r]^f \ar@{.>}[dr]^r \ar[ddr]_{1_{E_1}\times f} & E_2 \ar[dr]\\
    &P_B f\ar[r]\ar@{->>}[d] \pullbackcorner &
    P_B E_2\ar@{->>}[d]\\
    &E_1\times_B E_2\ar[r]_{f\times 1_{E_2}}&
    E_2\times_B E_2.
  }}\]
The same arguments applied in $\sM/B$ show that $f = p r$, where $r:E_1 \to P_B f$ and $p:P_B f \to E_2$ are a weak equivalence and a fibration respectively, both over $B$.
We now define
\[ \isequiv_B(f) \coloneqq \Pi_{p_2} \iscontr_{E_2}(P_B f), \]
where $p_2 : E_2 \fib B$ is the given fibration.
Note that this is an object of $\sM/B$.

\begin{lem}\label{thm:isequiv}
  For a map $f$ between fibrations $p_1:E_1 \fib B$ and $p_2:E_2\fib B$, the following are equivalent.
  \begin{enumerate}
  \item $f$ is a weak equivalence.\label{item:ie1}
  \item $\isequiv_B(f)\fib B$ has a section.\label{item:ie2}
  \item There is a map $1 \to \Pi_B \isequiv_B(f)$.\label{item:ie2a}
  \item $\isequiv_B(f)\fib B$ is an acyclic fibration.\label{item:ie3}
  \end{enumerate}
\end{lem}
\begin{proof}
  Our previous argument, applied in $\sM/B$, shows~\ref{item:ie1}$\Leftrightarrow$\ref{item:ie2}, and~\ref{item:ie2}$\Leftrightarrow$\ref{item:ie2a} is immediate by adjunction.
  And certainly~\ref{item:ie3}$\Rightarrow$\ref{item:ie2}, so it remains to show the converse.
  But by \autoref{thm:iscontr}, if $\iscontr_{E_2}(P_B f) \fib E_2$ has a section, then it is an acyclic fibration, and $\Pi_{p_2}$ preserves acyclic fibrations.
\end{proof}

\begin{lem}\label{thm:isequiv2}
  For a map $f$ between fibrations $p_1:E_1 \fib B$ and $p_2:E_2\fib B$, and a map $g:A\to B$, the following are equivalent.
  \begin{enumerate}
  \item The induced map $g^*E_1 \to g^* E_2$ is a weak equivalence.\label{item:ie2-1}
  \item $g$ lifts to $\isequiv_B(f)$.\label{item:ie2-2}
  \end{enumerate}
\end{lem}
\begin{proof}
  Just like the proof of \autoref{thm:iscontr2}.
\end{proof}

Our final goal now is to construct the universal \emph{space of} equivalences between two objects $E_1$ and $E_2$, or more generally between two fibrations $p_1 : E_1 \fib B$ and $p_2 : E_2\fib B$.
A logical place to start is with the universal space of \emph{functions}, namely the exponential $\fun_B(E_1,E_2)$ in $\sM/B$, which exists since $\sM$ is locally cartesian closed.
This object comes with a universal morphism $\fun_B(E_1,E_2) \times_B E_1 \to E_2$ over $B$.
By the universal property of a pullback, this universal morphism equivalently induces a morphism
\[ h:\fun_B(E_1,E_2) \times_B E_1 \too \fun_B(E_1,E_2) \times_B E_2 \]
over $\fun_B(E_1,E_2)$, which we will denote by $h$ as shown.
We think of $h$ as the ``universal family of functions $E_1 \to E_2$''.
If the objects of $\sM$ had ``points'', then the fiber of $h$ over a point $f\in \fun_B(E_1,E_2)$ (which itself would live over some point $b\in B$) would be the map $f : p_1^{-1}(b) \to p_2^{-1}(b)$ itself.
Formally, what this means is that for any object $A$, the bijection between lifts of $g:A\to B$ to $\fun_B(E_1,E_2)$ and morphisms $g^*E_1 \to g^* E_2$ over $A$ is implemented by pulling back $h$.

We can now construct $\isequiv_{\fun_B(E_1,E_2)}(h)$, which is a fibration over $\fun_B(E_1,E_2)$.
Its fiber over a ``point'' $f\in \fun_B(E_1,E_2)$ should be contractible if $f : p_1^{-1}(b) \to p_2^{-1}(b)$ is an equivalence, and empty otherwise.
Therefore, if we define
\[ \equiv_B(E_1,E_2) \coloneqq \Sigma_{\fun_B(E_1,E_2)} \isequiv_{\fun_B(E_1,E_2)}(h), \]
then $\equiv_B(E_1,E_2)$ (an object of $\sM/B$) will be the universal family of equivalences from $E_1$ to $E_2$ over $B$.
More precisely, to give a map $A\to \equiv_B(E_1,E_2)$ over some given map $g:A\to B$ is to give a map $A\to \fun_B(E_1,E_2)$ which lifts to $\isequiv_{\fun_B(E_1,E_2)}(h)$.
But by the universal property of $\fun_B(E_1,E_2)$ combined with \autoref{thm:isequiv2}, this is equivalent to giving a map $g^*E_1 \to g^*E_2$ which is an equivalence.

Finally, given a putative universe $p:\Util\to U$, we define the universal space of equivalences as:
\[\nEq(\Util) \coloneqq \equiv_{U\times U}(\pi_1^*\Util,\pi_2^*\Util).\]
This is an object of $\sM/(U\times U)$, with the property that lifting a map $(g_1,g_2):A\to U\times U$ to $\nEq(\Util)$ is equivalent to giving an equivalence $g_1^* \Util \to g_2 \Util$.
In fact, combining \autoref{thm:isequiv2} with the pullback-stability of local exponentials, we have
\[(g_1,g_2)^* \nEq(\Util) \cong \equiv_A(g_1^*\Util, g_2^*\Util) \]
Now we can prove \autoref{thm:eqlift}.

\begin{replem}{thm:eqlift}
  Let $\Ebar_1\fib B$ and $\Ebar_2\fib B$ be fibrations classified by maps $e_1,e_2:B\toto U$, let $\wbar\colon \Ebar_1\to \Ebar_2$ be a weak equivalence over $B$, let $i\colon A\cof B$ be a cofibration, and suppose we have a lift of $A \xto{i} B \to U\times U$ to $\nEq(\Util)$ which classifies $i^*(\wbar)$.
  Then this lift can be extended to $B$ so as to classify $\wbar$.
\end{replem}
\begin{proof}
  By the above remarks, we have $(e_1,e_2)^* \nEq(\Util) \cong \equiv_B(\Ebar_1,\Ebar_2)$, and so our given lift is equivalently a lift of $i$ to $\equiv_B(\Ebar_1,\Ebar_2)$.
  Let $k:B \to \fun_B(\Ebar_1,\Ebar_2)$ be the classifying map of $\wbar$; then we have the following commutative square:
  \begin{equation}
    \vcenter{\xymatrix{
        A\ar[r]\ar@{>->}[d]_i &
        \equiv_B(\Ebar_1,\Ebar_2)\ar@{->>}[d]\\
        B\ar[r]_-k &
        \fun_B(\Ebar_1,\Ebar_2)
      }}
  \end{equation}
  and hence also, invoking pullback-stability and the definition of $\equiv_B$, the following commutative square:
  \begin{equation}
    \vcenter{\xymatrix{
        A\ar[r]\ar@{>->}[d]_i &
        \isequiv_B(\wbar)\ar@{->>}[d]\\
        B\ar@{=}[r] \ar@{.>}[ur] &
        B.
      }}
  \end{equation}
  But since $\wbar$ is a weak equivalence, by \autoref{thm:isequiv} the right-hand fibration in this square is acyclic.
  Since $i$ is a cofibration, there exists a lift as shown, and tracing backwards this gives us our desired lift $B\to \nEq(\Util)$.
\end{proof}

\section{Modeling type theory}
\label{sec:modeling-type-theory}

In \S\ref{sec:on-univalence} we described a general plan for obtaining a universe object and showing that it is fibrant and univalent.
We remarked that for the interpretation of type theory, we need the small fibrations classified by our universe to be closed under the category-theoretic operations corresponding to all the basic type-forming operations: dependent sums, dependent products, and identity types.
We now show that this is the case for the universes we constructed in \S\ref{sec:on-univalence}.

The easiest case is dependent sums, which are modeled by composition of fibrations.
The composite of fibrations is always a fibration; for the composite of \ka-small fibrations to remain \ka-small we merely need \ka to be regular.

Dependent products are most directly modeled by right adjoints to pullback.
These exist in any locally cartesian closed category, but we require that the dependent product of a (\ka-small) fibration along a (\ka-small) fibration is again a (\ka-small) fibration.
The most natural way to ensure preservation of fibrations is via the adjoint condition that pullback along fibrations preserves acyclic cofibrations.
If the cofibrations are the monomorphisms, then they are stable under pullback, and if the model category is right proper, then weak equivalences are also stable under pullback along fibrations; so these two conditions together suffice.

For dependent product to preserve \ka-smallness in a presheaf category $\bSet^{\cC\op}$, we need \ka to be a strong limit cardinal and larger than $|\cC|$.
Thus, in conjunction with dependent sums, we need \ka to be inaccessible and larger than $|\cC|$.

Finally, the central insight of homotopy type theory is that identity types are modeled by path objects.
That is, for a \ka-small fibration $B\fib A$, we factor the diagonal $B\to B\times_A B$ into an acyclic cofibration followed by a fibration, $B \to P_A B \to B\times_A B$.
Since $B$ and $B\times_A B$ are fibrant in the slice model category over $A$, in a simplicial model category we can let $P_A B$ be the simplicial cotensor by $\Delta^1$ in this slice category.
In a category $\sSet^{\cC\op}$ of presheaves of simplicial sets, this preserves \ka-smallness as long as \ka is uncountable and larger than $|\cC|$.

There is also the issue of \emph{coherence} for all these structures, but fortunately this is taken care of by the coherence theorem of~\cite{lw:localuniv} or~\cite{awodey:natmodels}.
Thus, we can say:

\begin{thm}\label{thm:modeltt}
  If $\sSet^{\cC\op}$ has a right proper, cofibrantly generated simplicial model structure whose cofibrations are the monomorphisms, then it models type theory with dependent sums, dependent products, and intensional identity types.

  Moreover, if \ka is inaccessible and larger than $|\cC|$, and the codomains of the generating acyclic cofibrations are representable, then $\sSet^{\cC\op}$ contains a universe object classifying \ka-small fibrations and satisfying~\ref{item:u2p}.
  If this universe is fibrant (such as if~\ref{item:u3p} holds), it represents a univalent universe in the internal type theory.
\end{thm}

Of course, with multiple inaccessibles larger than $|\cC|$, we can find multiple universe objects of this sort, each contained in the next.
More precisely, if $\ka<\la$, then every \ka-small fibration is also \la-small, so we can find a pullback square
\begin{equation}
  \vcenter{\xymatrix@-.5pc{
      \Util\ar[r]\ar[d]_p &
      \Util'\ar[d]^{p'}\\
      U\ar[r] &
      U'
    }}
\end{equation}
where $p$ and $p'$ classify \ka-small and \la-small fibrations respectively, along with a classifying map $1\to U'$ for $U$ itself.

In fact, every proof of \autoref{thm:wellordered} gives us a little more: there is a canonical choice of such a pullback square in which $U \to U'$ is a monomorphism.
In the first proof, the inclusion $U \into U'$ is obvious.
In the second proof, we simply choose $\bSet_\la$ so that it contains $\bSet_\ka$.
And in the third proof, we can inductively construct monomorphisms $U_\alpha \into U'_\alpha$ which are preserved by all the colimits, as long as we choose the sets of $(i,f,p)$s for $U'$ to contain those for $U$.

These canonical inclusions are important, because in order to model a cumulative hierarchy of universes in type theory, we need to ensure furthermore that the inclusions respect the ``universe structure''.
To explain what this means, suppose $\Util\fib U$ is a universe.
Then the local exponential
\begin{equation}
  U^{(1)} \coloneqq (U\times U \to U)^{(\Util\to U)}\label{eq:udnt}
\end{equation}
is the base of a universal pair of composable small fibrations.
If the composite of these fibrations is again small, it is classified by some map $\Sigma:U^{(1)} \to U$, and in order to model a type-theoretic universe by $U$ we must \emph{choose} such a map.

Similarly, if the dependent product of the universal composable pair of small fibrations is small, we can choose for it a classifying map $\Pi:U^{(1)} \to U$.
And for identity types, we consider $\Util\times_U \Util$, which is the base of a universal ``type with two sections''.
We have a small fibration $P_U \Util \fib \Util\times_U \Util$, where $P_U \Util$ is the path object of $\Util$ in $\sSet^{\cC\op}/U$, and we can choose for it a classifying map $\mathrm{Id}:\Util\times_U \Util \to U$.

The requirement for a cumulative hierarchy of universes is then that the inclusions $U \into U'$ respect this chosen structure.
In~\cite{shulman:invdia} I called such a $U\into U'$ a \emph{universe embedding}.
Fortunately, if our universes all satisfy property~\ref{item:u2p}, then any inclusion can be made into a universe embedding as follows.

Suppose, under the hypotheses of \autoref{thm:modeltt}, that
we have a monomorphism $i\colon U\cof U'$ between universes such that $i^*(\Util') \cong \Util$.
Then there is an induced monomorphism $U^{(1)}\cof (U')^{(1)}$, which pulls back the universal composable pair of $U'$-small fibrations to the analogous pair of $U$-small ones.
If in addition we have chosen some classifying map $U^{(1)}\to U$ for the universal composite of $U$-small fibrations, then composing it with $i$ we obtain another classifying map $U^{(1)}\to U'$ for the same fibration.
But now since $U'$ satisfies~\ref{item:u2p}, we can extend this to a compatible classifying map $(U')^{(1)} \to U'$ for the universal composite of $U'$-small fibrations.

Thus, given $\Sigma: U^{(1)} \to U$, it is always possible to choose $\Sigma':(U')^{(1)} \to U'$ such that $i$ commutes with $\Sigma$ and $\Sigma'$.
The same technique applies to dependent products, and also to path objects as long as we choose constructions of the universal path objects relative to $U$ and $U'$ which are compatible under $i^*$ (up to isomorphism); in our simplicial model category, we can again use the cotensor with $\Delta^1$.
We can furthermore induct up any sequence of universe inclusions
\[ U_0 \cof U_1 \cof U_2 \cof \cdots \]
to obtain a sequence of universe embeddings.
Thus we have:

\begin{thm}\label{thm:modeltt2}
  Under all the hypotheses of \autoref{thm:modeltt}, $\sSet^{\cC\op}$ contains as many nested universe objects satisfying~\ref{item:u2p} as there are inaccessible cardinals greater than $|\cC|$.
  If these universes are fibrant, they represent univalent universes in the internal type theory.
\end{thm}

Instead of the coherence theorem of~\cite{lw:localuniv}, we can also use the method of~\cite{klv:ssetmodel}, with the universes themselves providing coherence.
However, this requires either the use of an ``outer'' universe providing the coherence, so that there is one fewer universe in the type theory than there are universes in the model category, or an infinite hierarchy of universes.

\section{Elegant Reedy presheaves}
\label{sec:elegant}

Finally, we show that~\ref{item:u3p} holds in the Reedy model structure of simplicial presheaves on any elegant Reedy category.
This completes the proof that type theory with univalent universes can be interpreted in such model categories.

Recall that for any simplicial category $\cC$, there is an \emph{injective model structure} on the category $\sSet^{\cC\op}$ of simplicial presheaves which is cofibrantly generated, left and right proper, and simplicial, and its cofibrations are the monomorphisms.
Thus, it satisfies all the conditions of Theorems~\ref{thm:modeltt} and~\ref{thm:modeltt2} except for representability of the codomains of the generating acyclic cofibrations.
In fact, in general the generating acyclic cofibrations are the most mysterious part of the injective model structure.

However, there is a special class of categories $\cC$ for which the injective model structure can be described much more explicitly.
When $\cC$ is an \emph{elegant Reedy category} as in~\cite{br:reedy}, the injective model structure coincides with the \emph{Reedy model structure}, which is perhaps the most explicit sort of model structure that can be put on a category of simplicial presheaves.
We will show that in this case, the rest of the structure follows as well, so that $\sSet^{\cC\op}$ models type theory with univalence.

Recall from~\cite[Ch.~15]{hirschhorn:modelcats} or~\cite[Ch.~5]{hovey:modelcats} that \cC is a \emph{Reedy category} if the following hold.
\begin{itemize}[leftmargin=2em]
\item There is a well-founded relation $\prec$ on the objects of $\cC$.
\item There are two subcategories $\cC^+$ and $\cC^-$ containing all the objects of $\cC$.
\item Every morphism $\al$ of $\cC$ can be written uniquely as $\al^+\al^-$, where $\al^+$ lies in $\cC^+$ and $\al^-$ lies in $\cC^-$.
\item If $\al\colon c\to d$ lies in $\cC^+$ and is not an identity, then $c\prec d$.
\item If $\al\colon c\to d$ lies in $\cC^-$ and is not an identity, then $d\prec c$.
\end{itemize}
We say \cC is \emph{direct} if $\cC^-$ contains only identities, and \emph{inverse} if $\cC^+$ contains only identities.

For a presheaf $X\in \sSet^{\cC\op}$ on a Reedy category \cC and an object $c\in \cC$, the \emph{matching object} is defined by
\[ M_c X \coloneqq \operatorname{lim}_{\partial(\cC^+ \dn c)\op} \left(X|_{\partial(\cC^+\dn c)\op}\right) \]
where $\partial(\cC^+ \dn c)$ denotes the full subcategory of the over-category $(\cC^+ \dn c)$ which omits the identity arrow of $c$.
Similarly, the \emph{latching object} is defined by
\[ L_c X \coloneqq \colim_{\partial(c\dn \cC^-)\op} \left(X|_{\partial(c\dn \cC^-)\op}\right). \]
Then the category $\sSet^{\cC\op}$ has a model structure, called the \emph{Reedy model structure}, in which a morphism $f\colon A\to B$ is a fibration just when each map
\begin{equation}
  A_c \too B_c \times_{M_c B} M_c A\label{eq:rfibmap}
\end{equation}
is a fibration in \sSet, a cofibration just when each map
\begin{equation}
  A_c \amalg_{L_c A} L_c B \too B_c\label{eq:rcofibmap}
\end{equation}
is a cofibration in \sSet, and a weak equivalence just when it is a levelwise weak equivalence in \sSet.
This model structure is simplicial and left and right proper. 
For a categorical perspective on this construction, see~\cite{rv:reedy}.

Note that if \cC is direct, then each $L_c A$ is initial, so the Reedy cofibrations are just the levelwise ones, and dually.
For future use, we record the following:

\begin{lem}\label{thm:latchcof}
  If $f:A\to B$ is a Reedy cofibration, then each map $L_c f\colon  L_c A \to L_c B$ is a cofibration, which is acyclic if $f$ is.
  Dually, if $f$ is a Reedy fibration, then each $M_c f\colon  M_c A \to M_c B$ is a fibration, which is acyclic if $f$ is.
\end{lem}
\begin{proof}
  Since $\partial(c\dn \cC^-)$ is an inverse category, its Reedy fibrations are levelwise.
  In particular, the constant diagram functor $\sSet \to \sSet^{\partial(c\dn \cC^-)\op}$ is right Quillen, and so the colimit functor over $\partial(c\dn \cC^-)\op$ is left Quillen.
  Hence, it suffices to show that the restriction functor $\sSet^{\cC\op} \to \sSet^{\partial(c\dn \cC^-)\op}$ preserves Reedy cofibrations.
  But given $\gm\colon c\to d$ in $\partial(c\dn \cC^-)$, we have $\partial(\gm \dn \partial(c\dn \cC^-)) \cong \partial(d\dn \cC^-)$, so this restriction preserves latching objects.

  An alternative argument, showing that $L_c f$ is a cell complex whose cells are the maps~\eqref{eq:rcofibmap} at objects preceding $c$, can be found in the third paragraph of the proof of~\cite[Lemma 7.1]{rv:reedy}.
\end{proof}

\begin{rmk}\label{rmk:latchcof}
  \autoref{thm:latchcof} remains true even if $A$, $B$, and $f$ are only defined on the full subcategory of objects $d\in \cC$ with $d\prec c$.
  Thus, we can use it during the standard Reedy process of building up diagrams and maps inductively.
\end{rmk}

By~\cite[15.6.24]{hirschhorn:modelcats} or~\cite[7.7]{rv:reedy}, the Reedy model structure on $\sSet^{\cC\op}$ is cofibrantly generated; the generating Reedy acyclic cofibrations are the pushout products
\begin{equation}
  (\Lambda^n_k \otimes Y_c) \cup_{\Lambda^n_k \otimes L_c Y} (\Delta^n \otimes L_c Y) \too \Delta^n \otimes Y_c .\label{eq:reedy-gen}
\end{equation}
Here $Y\colon \cC \to \sSet^{\cC\op}$ is the Yoneda embedding, and $K\otimes X$ denotes the simplicial tensor (which in this case is the levelwise cartesian product).
Since the functors $\Delta^n \otimes Y_c$ are exactly the representable functors in $\sSet^{\cC\op}$ (when regarded as the presheaf category $\bSet^{\Delta\op\times \cC\op}$), we can apply \autoref{thm:wellordered} to obtain universes for small Reedy fibrations that satisfy~\ref{item:u2p}.

Note that for any regular cardinal \ka, if $f$ and $g$ are composable functions such that $g$ has \ka-small fibers, then $f$ has \ka-small fibers if and only if $g f$ does.
Moreover, \ka-small morphisms are closed under limits of size $<\ka$.
Thus, if $\ka > |\cC|$, a Reedy fibration is \ka-small in the sense of \S\ref{sec:on-univalence} if and only if each map~\eqref{eq:rfibmap} is a \ka-small fibration in \sSet.

We henceforth assume \cC to be an \emph{elegant Reedy category}.
This is a combinatorial condition due to~\cite{br:reedy} which ensures that the Reedy cofibrations in $\sSet^{\cC\op}$ are exactly the (levelwise) monomorphisms, i.e.\ that the Reedy model structure coincides with the injective one.
Examples of elegant Reedy categories include direct categories, the simplex category $\Delta$, the $n$-fold simplex category $\Delta^n$, and Joyal's categories $\Theta_n$.
Thus, to obtain a model of type theory with univalent universes in $\sSet^{\cC\op}$, it remains only to show that the universes are fibrant.

\begin{lem}\label{thm:reedy-3p}
  If $\cC$ is an elegant Reedy category, then the Reedy model structure on $\sSet^{\cC\op}$ satisfies~\ref{item:u3p}.
\end{lem}
\begin{proof}
  Let $i\colon A\cof B$ be a Reedy (i.e.\ levelwise) acyclic cofibration, and let $p\colon P\fib A$ be a small Reedy fibration.
  The small fibration $Q\to B$ we want to define will be, in particular, a factorization of the composite $p i\colon P \to B$.
  By a standard argument for Reedy diagrams (e.g. as found in~\cite[5.2.5]{hovey:modelcats},~\cite[15.3.16]{hirschhorn:modelcats}, or~\cite[7.4]{rv:reedy}), to give such a factorization is equivalent to giving, by well-founded induction on $c\in \cC$, a factorization of the induced map
  \[ P_c \amalg_{L_c P} L_c Q \too B_c \times_{M_c B} M_c Q \]
  (with $L_c Q$ and $M_c Q$ being defined inductively as we go).
  Equivalently, we must give an object $Q_c$ and dashed arrows which complete the following diagram to be commutative:
  \begin{equation}
    \vcenter{\xymatrix{
        L_c P\ar[r]\ar[d] &
        L_c Q\ar@{-->}[d] \ar@(dr,ur)[dd]\\
        P_c \ar@{-->}[r]\ar[d] &
        Q_c \ar@{-->}[d]\\
        A_c \times_{M_c A} M_c P\ar[r] &
        B_c \times_{M_c B} M_c Q.
      }}\label{eq:r3prect}
  \end{equation}
  By assumption, the lower-left map $P_c \to A_c \times_{M_c A} M_c P$ in~\eqref{eq:r3prect} is a small fibration.
  Thus, it has a classifying map $A_c \times_{M_c A} M_c P \to U$, where $U$ is a universe for \ka-small Kan fibrations in \sSet.
  (Recall that this universe is fibrant, by~\cite[Theorem 2.2.1]{klv:ssetmodel}.)
  The section $L_c P \to P_c$ of this fibration corresponds to a dashed lifting:
  \[\xymatrix{ & & \Util \ar@{->>}[d]\\
    L_c P \ar[r] \ar@{-->}[urr] & A_c \times_{M_c A} M_c P \ar[r] & U.
  }\]
  Now by induction, for all $\al\colon c\to d$ in $\cC^-$, the map $P_d \to Q_d$ is a pullback of an acyclic cofibration along a fibration, hence also an acyclic cofibration.
  Thus, the map $P\to Q$ is (insofar as it is defined) a levelwise acyclic cofibration, hence (by elegance) a Reedy acyclic cofibration.
  By \autoref{thm:latchcof} (and \autoref{rmk:latchcof}), therefore, $L_c P \to L_c Q$ is an acyclic cofibration.
  Since \Util is fibrant, we can thus extend the above classifying map $L_c P \to \Util$ to $L_c Q$.

  Now I claim that the bottom horizontal map in~\eqref{eq:r3prect} is an acyclic cofibration.
  By induction, for all $\al\colon d\to c$ in $\cC^+$, we have a pullback square
  \[\vcenter{\xymatrix@-.5pc{
      P_d\ar[r]\ar[d] \pullbackcorner &
      Q_d\ar[d]\\
      A_d\ar[r] &
      B_d.
    }}\]
  Therefore, the right-hand square below is also a pullback (while the left-hand square is a pullback by definition):
  \[\vcenter{\xymatrix{
      A_c \times_{M_c A} M_c P \ar[r] \ar[d] \pullbackcorner &
      M_cP\ar[r]\ar[d] \pullbackcorner &
      M_cQ\ar[d]\\
      A_c \ar[r] &
      M_cA\ar[r] &
      M_cB.
    }}\]
  Thus, the outer rectangle above is also a pullback.
  Since this is also the outer rectangle in the next diagram, whose right-hand square is a pullback by definition, so is its left-hand square.
  \begin{equation}
  \vcenter{\xymatrix{
      A_c \times_{M_c A} M_c P \ar[r] \ar[d] \pullbackcorner &
      B_c \times_{M_c B} M_c Q\ar[r]\ar[d] \pullbackcorner &
      M_c Q\ar[d]\\
      A_c \ar[r] &
      B_c \ar[r] &
      M_cB.
    }}\label{eq:r3rect2}
  \end{equation}
  But $Q \to B$ is (insofar as it has been defined) a Reedy fibration; hence by \autoref{thm:latchcof}, $M_c Q \to M_c B$ is a fibration.
  Therefore, so is the middle vertical map in~\eqref{eq:r3rect2}.
  This means that the left-hand square in~\eqref{eq:r3rect2} exhibits the bottom horizontal map in~\eqref{eq:r3prect} as a pullback of the acyclic cofibration $A_c \to B_c$ along a fibration, so it is an acyclic cofibration.

  Let $D$ be the following pushout, with induced map as shown:
  \[\vcenter{\xymatrix{
      L_c P\ar[r]\ar[d] &
      L_c Q\ar[d] \ar[ddr] \\
      A_c \times_{M_c A} M_c P\ar[r] \ar[drr] &
      D \pushoutcorner \ar@{.>}[dr]\\
      && B_c \times_{M_c B} M_c Q.
    }}\]
  Since every morphism in $\cC^-$ is split epic in an elegant Reedy category, every morphism $L_c X \to M_c X$ is a monomorphism.
  It follows that the maps $L_c P \to A_c \times_{M_c A} M_c P$ and $L_c Q \to  B_c \times_{M_c B} M_c Q$ are also monomorphisms.
  We have already observed that $L_c P \to L_c Q$ and $A_c \times_{M_c A} M_c P\to B_c \times_{M_c B} M_c Q$ are monomorphisms (cofibrations), so the above pushout is a union of subobjects, and hence the induced dotted map is also a monomorphism.

  Moreover, we have also observed that $A_c \times_{M_c A} M_c P\to B_c \times_{M_c B} M_c Q$ is an acyclic cofibration, and so is $A_c \times_{M_c A} M_c P\to D$ since it is a pushout of such.
  Therefore, by the 2-out-of-3 property, the induced dotted map is also a weak equivalence, hence an acyclic cofibration.

  Now recall that we have a classifying map $A_c \times_{M_c A} M_c P \to U$ for $P_c$, and an extension to $L_c Q$ of its restriction to $L_c P$ (by way of \Util).
  Thus, we have an induced map $D\to U$, and since $U$ is fibrant we can extend this map to $B_c \times_{M_c B} M_c Q$.
  Let $Q_c \to B_c \times_{M_c B} M_c Q$ be the fibration classified by this map.
  Then we have all the dashed arrows making~\eqref{eq:r3prect} commutative and its lower square a pullback.
  By pasting this on top of the left-hand square in~\eqref{eq:r3rect2}, we see that $P_c$ is the pullback of $Q_c$ along $i_c$, as desired.
\end{proof}

Putting this together with \S\S\ref{sec:on-univalence}--\ref{sec:modeling-type-theory}, we have shown:

\begin{thm}
  For any elegant Reedy category \cC, the Reedy model category $\sSet^{\cC\op}$ supports a model of intensional type theory with dependent sums and products, identity types, and as many univalent universes as there are inaccessible cardinals greater than $|\cC|$.\qed
\end{thm}

Since direct categories are elegant Reedy categories, and presheaves on a direct category are of course the same as covariant diagrams on its opposite (which is an inverse category), this generalizes (the restriction to $\sSet$ of) the corresponding theorem proven in~\cite{shulman:invdia}.

It also includes some new examples, such as the model categories $\sSet^{\Delta\op}$ of bisimplicial spaces and $\sSet^{\Theta_n\op}$ of $\Theta_n$-spaces.
These model categories are interesting, among other reasons, because they have localizations that present theories of higher categories~\cite{rezk:css,rezk:cpncats}.
The univalent universes we have constructed in these model categories have ``sub-universes'' corresponding to these localizations, which may be useful for applying type theory to the study of higher categories.

\bibliographystyle{alpha}
\bibliography{all}

\end{document}